\documentclass{article}
\usepackage{amsmath,amsthm,amssymb,amsfonts}
\usepackage[latin1]{inputenc}
\usepackage{graphics,graphicx}
\usepackage{hyperref}
\evensidemargin = \oddsidemargin
\theoremstyle{plain}
\newtheorem{theorem}{Theorem}
\newtheorem*{theorem'}{Theorem A'}
\newtheorem{proposition}{Proposition}
\newtheorem{lemma}[proposition]{Lemma}
\theoremstyle{remark}

\def\norm#1{\left\lVert #1 \right\rVert}

\def\N{{\mathbb N}}

\def\R{{\mathbb R}}
\def\Z{{\mathbb Z}}
\def\CC{{\mathcal C}}
\def\Cinf{{\mathcal C^\infty}}
\def\DD{{\mathcal D}}

\def\ZZ{{\mathcal Z}}
\def\id{{\mathrm{id}}}
\def\ly{\fontencoding{U}\fontfamily{lasy}\fontseries{m}\fontshape{n}\selectfont}
\def\guil#1{\leavevmode\hbox{{\ly(\kern-0.20em(\kern+0.20em}}\nobreak{}\,#1\,%
  \nobreak\hbox{{\ly\kern+0.20em)\kern-0.20em)}}}
\def\up{\textup}
\def\from{\colon} 
 
\def\bigbars#1{\bigl\lvert #1 \bigr\rvert} 
 
\def\lrbars#1{\left\lvert #1 \right\rvert} 
\def\Bars#1{\lVert #1 \rVert} 
\def\bigBars#1{\bigl\lVert #1 \bigr\rVert} 
 
\def\lrBars#1{\left\lVert #1 \right\rVert} 
\def\Supp{\mathop{\mathrm{Supp}}}

\def\res#1{\mathbin{|}{}_{#1}}

\begin{document}
\title{Arithmetic properties of centralizers of diffeomorphisms of the half-line}
\author{\begin{tabular}{c}
{Hélène \textsc{Eynard}}\\
\\
{\small{Institut de Mathématiques de Jussieu (UMR 7586)}} \\
{\small{Université Pierre et Marie Curie}}\\
{\small{4 place Jussieu, 75252 Paris Cedex 5, France}}\\
{\small\texttt{heynardb@math.jussieu.fr}}
\end{tabular}
}

\date{}

\maketitle

\begin{abstract}
Let $f$ be a smooth diffeomorphism of the half-line fixing only the origin and
$\ZZ^r_f$ its centralizer in the group of $\CC^r$ diffeomorphisms. According to
well-known results of Szekeres and Kopell, $\ZZ^1_f$ is always a one-parameter
group, naturally identified to $\R$, with $f \cong 1$. On the other hand,
$\ZZ^r_f$, $2 \le r \le \infty$, can be smaller: in \cite{Se}, Sergeraert
constructed an $f$ whose $\Cinf$ centralizer reduces to the infinite cyclic
group generated by $f$ (\emph{i.e} $\ZZ^\infty_f \cong \Z$). In \cite{Ey1}, we
adapted Sergeraert's construction to obtain an $f$ whose $\CC^r$ centralizer,
for all $2 \le r \le \infty$, contains a Cantor set $K$ but is still strictly
smaller than $\ZZ^1_f \cong \R$. Here, we improve \cite{Ey1}
to construct, for any Liouville number $\alpha$, an $f$ as above such
that, in addition, $\alpha \in K \subset \ZZ^r_f$.
\end{abstract}

\newpage

We want to understand what the $\CC^r$ centralizer, $2 \le r \le \infty$, of a
smooth ($\Cinf$) diffeomorphism $f$ of $\R_+ = [0,\infty)$ can possibly look like. If
$\DD^r$ denotes the group of $\CC^r$ diffeomorphisms of $\R_+$, $1 \le r \le
\infty$, endowed with the usual $\CC^r$ (compact-open) topology, the $\CC^r$
centralizer $\ZZ^r_f$ of $f$ is the (closed) subgroup of $\DD^r$ made up of all
diffeomorphisms commuting with $f$. Here, we limit ourselves to
diffeomorphisms $f$ which fix only the origin. The $\CC^1$ centralizer of such
an $f$ is very well understood: well-known theorems by G. Szekeres and N. Kopell
\cite{Sz,Ko} show that $\ZZ^1_f$ is always a one-parameter subgroup of $\DD^1$
(see also \cite[chap. 4]{Yo} and \cite[chap. 4]{Na} for complete proofs and more
discussion). More precisely, $f$ is the time-$1$ map of a unique $\CC^1$ vector
field $\nu_f$ on $\R_+$ (we call it the \emph{Szekeres vector field of $f$}),
and $\ZZ^1_f$ reduces to the flow of $\nu_f$. Hence, there is a natural
identification of $\ZZ^1_f$ to $\R$, with $f \cong 1$. Since $\ZZ_f^r$ decreases
with $r$ and contains the infinite cyclic subgroup generated by $f$, one has
\begin{equation*}
\Z \cong \{ f^n,\ n \in \Z \} \quad \subset \quad \ZZ^r_f \quad \subset \quad 
\ZZ^1_f  \cong \R. 
\end{equation*}
If $\nu_f$ is of class $\CC^r$, the inclusion on the right is an equality.
According to F. Takens \cite{Ta}, this is always the case if $f$ is not
infinitely tangent to the identity at $0$. However, this inclusion can also be
strict, as Sergeraert shows in \cite{Se}, and one can actually check~\cite{Ey2}
that in his example, $\ZZ^2_f = \ZZ^\infty_f$ reduces to the group spanned
by~$f$, and is hence as small as possible. It is then easy, for any integer
$q \ge 1$, to find an $f$ whose $\CC^\infty$ centralizer, seen as a subgroup of
$\R$, is $\frac 1 q \Z$. The next natural question then is whether $\ZZ^\infty_f$
can be a dense (but still proper) subgroup of $\ZZ_f^1 \cong \R$. Article
\cite{Ey1} gives a positive answer: $\ZZ^\infty_f$ can contain a Cantor set $K$.

In the construction of \cite{Ey1}, based on Sergeraert's techniques and
Anosov--Katok-like methods (introduced in \cite{AK}; see also \cite{FK} and the
references therein), the very good approximation of all elements of $K $ by
rational numbers plays a crucial role. This fact urges us to consider
$\ZZ^\infty_f$ not merely from a topological point of view, but from an
arithmetic one:
\begin{center} 
\emph{What kind of irrational numbers can $\ZZ^\infty_f$ contain ?}
\end{center} 
\noindent Here, it seems natural to distinguish between numbers which satisfy a
diophantine condition (\emph{i.e} are ``badly" approximated by rational numbers)
and numbers which do not. Recall that a number $\alpha$ is said to \emph{satisfy a
diophantine condition} if there exist constants $c > 0$ and $\gamma \ge 0$ such that
\begin{equation}
\lrbars{\alpha - \frac p q} \ge c q^{- 2 - \gamma}
\end{equation}
for every rational number $p/q$, with $q \ge 1$. An irrational number which
satisfies no diophantine condition is called a \emph{Liouville number}. The
following result might constitute one half of an answer to the above question.

\begin{theorem}\label{t:principal}
For any Liouville number $\alpha$, there exists a $\CC^\infty$ diffeomorphism
$f$ of $\R_+$ with a single fixed point at the origin, whose $\CC^r$
centralizer, for all $2 \le r\le \infty$, is a proper subgroup of $\ZZ^1_f \cong
\R$ and contains a Cantor set $K \ni \alpha$.
\end{theorem}

The aim of this article is to prove the following equivalent statement.
\begin{theorem'}\label{t:principalb}
For any Liouville number $\alpha$, there exists a $\CC^1$ vector field $\nu$ on
$\R_+$ vanishing only at $0$ whose time-$t$ map is smooth for every
$t \in \{1\} \cup K$, for some Cantor set $K$ containing $\alpha$, but not $\CC^2$
for some other $t \in \R$.
\end{theorem'}

Half of the question remains open: one would now like to prove that a $\CC^1$
vector field on $\R_+$ whose time-$1$ and $\alpha$ maps are smooth, for some
$\alpha$ satisfying a diophantine condition, is necessarily smooth itself,
drawing one's inspiration from similar problems in the case of circle
diffeomorphisms. This parallel suggests many more questions: can the set of
smooth times be dense but countable? Is there some particular arithmetic
relation between two irrational smooth times of a nonsmooth flow?... \medskip

\noindent \textbf{Acknowledgements.} I wish to thank Jean-Christophe Yoccoz and
Bassam Fayad for sharing their insight on the subject with me and encouraging me
to investigate the particular question which motivates this article. I am also
very greatful to Sylvain Crovisier who suggested (and helped me understand) the
dynamical techniques at stake here. However, none of this would have ever
materialized without Emmanuel Giroux's precious advice and endless will
to understand things better, which had a tremendous (though partly delayed)
effect on my own grasp of the subject. I deeply thank him for this gift.

The above exchanges were made easier by the financial support of the Agence
Nationale de la Recherche through the ``Symplexe" project. The article itself
was written during a one year stay in Tokyo where I was kindly invited by Pr.
Takashi Tsuboi, with the financial support of the Japan Society for the
Promotion of Science.

\section{Overview}
\label{s:overview}

The general idea of the construction is the same as in \cite{Ey1}. We repeat it
here for completeness' sake as well as to emphasize the slight (but key)
improvements, gathered at the end of the section. All statements will be
made precise and proved afterwards, in Sections \ref{s:manu} to \ref{s:conv}.

\subsection{Sergeraert's construction}

We first need to explain how to build a $\CC^1$ vector field whose flow is
smooth for some times but not $\CC^2$ for others. To that end, we sketch
Sergeraert's construction (with some minor modifications). Sergeraert starts
with a diffeomorphism $f_0$ which is the time-$1$ map of a ``well-chosen" smooth
vector field $\nu_0$ on $\R_+$ (described later). He subjects it to infinitely
many ``small" (explicit) perturbations, with disjoint supports, closer and
closer to $0$, denoted by $\gamma_k$, $k \in \N^* = \N \setminus \{0\}$, so that
$$f = f_0 + \sum_{k\ge 1} \gamma_k$$
is still a smooth diffeomorphism of $\R_+$ (to ensure this, he only needs to
pick the $\gamma_k$'s so that their sum converges in $\Cinf$ topology and is
$\CC^1$-small compared to $f_0$), but that its Szekeres vector field, on the
other hand, is not smooth anymore. More precisely, he makes sure that the
time-$1/2$ map of the resulting vector field is not $\CC^2$.

It is not straightforward, even when one knows their expressions, to visualize
the effect of the perturbations $\gamma_k$ on the Szekeres vector field of $f_0$
and on its time-$1/2$ map. A way to understand how things work is to interprete
Sergeraert's construction in terms of deformation by conjugation. Let us
therefore describe the construction all over again, in a different language.

We start with the same smooth vector field $\nu_0$ (Sergeraert's, described
below) and this time, we are going to obtain the desired vector field $\nu$ (the
one with a smooth time-$1$ map and a non $\CC^2$ time-$1/2$ map) as a limit of a
sequence of deformations $\nu_k$, each $\nu_k$ being the pull-back $h_k^*\nu_0$
of $\nu_0$ by a smooth diffeomorphism $h_k$ of $\R_+$. The flow $f_k^t$ of
$\nu_k$ is then related to the flow $f_0^t$ of $\nu_0$ by $f_k^t = h_k^{-1}
\circ f_0^t \circ h_k$. The point is to cook up the conjugations $h_k$ so that
$f_k^1$ converges in $\Cinf$ topology while $f_k^{1/2}$ converges only in
$\CC^1$ topology (in particular, the $h_k$ must diverge in $\CC^2$ topology).

Here, the behaviour of the initial vector field plays a crucial role: it
vanishes only at $0$, is negative elsewhere, and its graph resembles an undersea
landscape consisting of a sequence of alternating lowlands $L_n$ and highlands
$H_n$, accumulating at the origin, whose respective altitudes $-v_n$ and $-u_n$
(measured from the water surface, so that $0 < u_n < v_n$) go to zero very fast
when $n$ grows (so that $\nu_0$ is infinitly flat at $0$), but ``oscillate
wildly'' in the sense that the ratios $ v_n / u_n $ (and actually $v_n^k / u_n$ for 
all $k$) tend to infinity.
\begin{figure}[htbp]
\centering
\includegraphics[width=12cm,height=1.2cm]{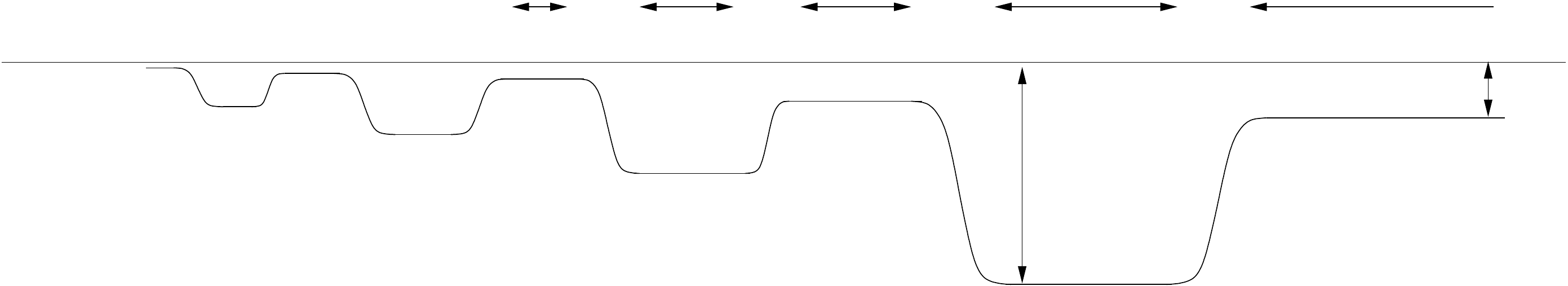}
\put(-118,14){$\scriptstyle{v_n}$}
\put(-15,21){$\scriptstyle{u_n}$}
\put(-74,10){$\scriptstyle{\nu_0}$}
\put(-48,37){$\scriptstyle{H_{n}}$}
\put(-110,37){$\scriptstyle{L_n}$}
\put(-165,37){$\scriptstyle{H_{n+1}}$}
\put(-201,37){$\scriptstyle{L_{n+1}}$}
\put(-234,37){$\scriptstyle{H_{n+2}}$}
\label{fig:undersea}
\end{figure}
A consequence of this behaviour is that, if an element $f_0^t$ of the flow takes
a segment $S \subset L_n$ (resp. $S \subset H_n$) into $L_n$, then its
restriction to $S$ is the translation $x \mapsto x - tv_n$ (resp. an affine map
with big dilation factor $ v_n / u_n $). This follows immediatly from the
invariance of $\nu_0$ under its flow: $\nu_0 \circ f_0^t = \nu_0 \times Df_0^t$.

In the light of these remarks, we can move on to the definition of the
conjugations $h_k$. What we actually construct for each $k$ is a diffeomorphism
$g_k$, and we then define $h_k$ as $g_k \circ h_{k-1}$. Hence $\nu_k =
h_k^*\nu_0 = h_{k-1}^*g_k^*\nu_0$, so that the flows of $\nu_k$ and $\nu_{k-1}$
are given by
\begin{align*}
f_k^t = h_{k-1}^{-1} \circ (g_k^{-1} & \circ f_0^t \circ g_k) \circ h_{k-1}
\quad \text{and} \\
f_{k-1}^t = h_{k-1}^{-1} & \circ f_0^t \circ h_{k-1}
\end{align*}  
respectively. Thus, intuitively, we want $g_k^{-1} \circ f_0^1 \circ g_k -
f_0^1$ to be $\CC^k$-small (say less than $2^{-k}$) while $g_k^{-1} \circ
f_0^{1/2} \circ g_k - f_0^{1/2}$ is $\CC^2$-big. To do that, we chose a $g_k$
which
\begin{itemize}
\item commutes with $f_0^1$ everywhere except in a small region: a fondamental
interval $S_k$ of $f_0^1$ lying ``in the middle of $L_k$";
\item is $\CC^k$ close to the identity in this region.
\end{itemize}
More precisely, we take $g_k$ equal to the identity near $0$ and of the form
$\id + \gamma_k$ on $S_k$, where $\gamma_k$ is a $\CC^k$ small function
supported in $S_k$, of the form:
\begin{figure}[h!]
\centering
\includegraphics[width=6cm]{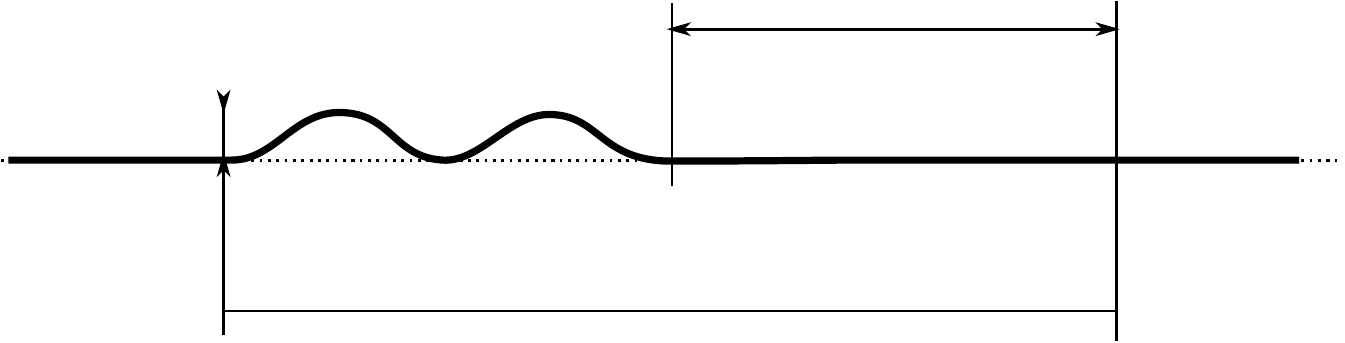}
\put(-90,-5){$S_k$}
\put(-130,36){$\gamma_k$}
\put(-155,26){$u_k$}
\put(-65,45){$v_k/2$}
\label{fig:gammak}
\end{figure}

\noindent(we will see shortly why this form in particular).
One easilly checks that this choice of $g_k$ gives:
$$f_k^1 = f_{k-1}^1 + \gamma_k.$$ 
(this construction is thus really equivalent to Sergeraert's). The support of
$g_k - \id$, on the other hand, is \emph{not} $S_k$. Indeed, the above
information is enough to determine $g_k$ on all of $\R_+$: $g_k$ is the identity
on $[0, \min S_k]$, but $[\max S_k, + \infty)$ is tiled by segments $S_k^p =
f_0^{- p/{q_k}} (S_k)$, $p \ge 1$, on which
\begin{equation*}
g_k \res{S_k^p} \quad = \quad f_0^{-p} \circ (g_k\res{S_k}) \circ f_0^p \quad =
\quad f_0^{-p} \circ (\id + \gamma_k) \circ f_0^p.
\end{equation*}
On $[\sup S_k,\sup L_k]$ in particular, $f_0^1$ coincides with the translation
by $-v_k$, so $g_k$ commutes with this translation.
\begin{figure}[h!]
\centering
\includegraphics[width=6cm]{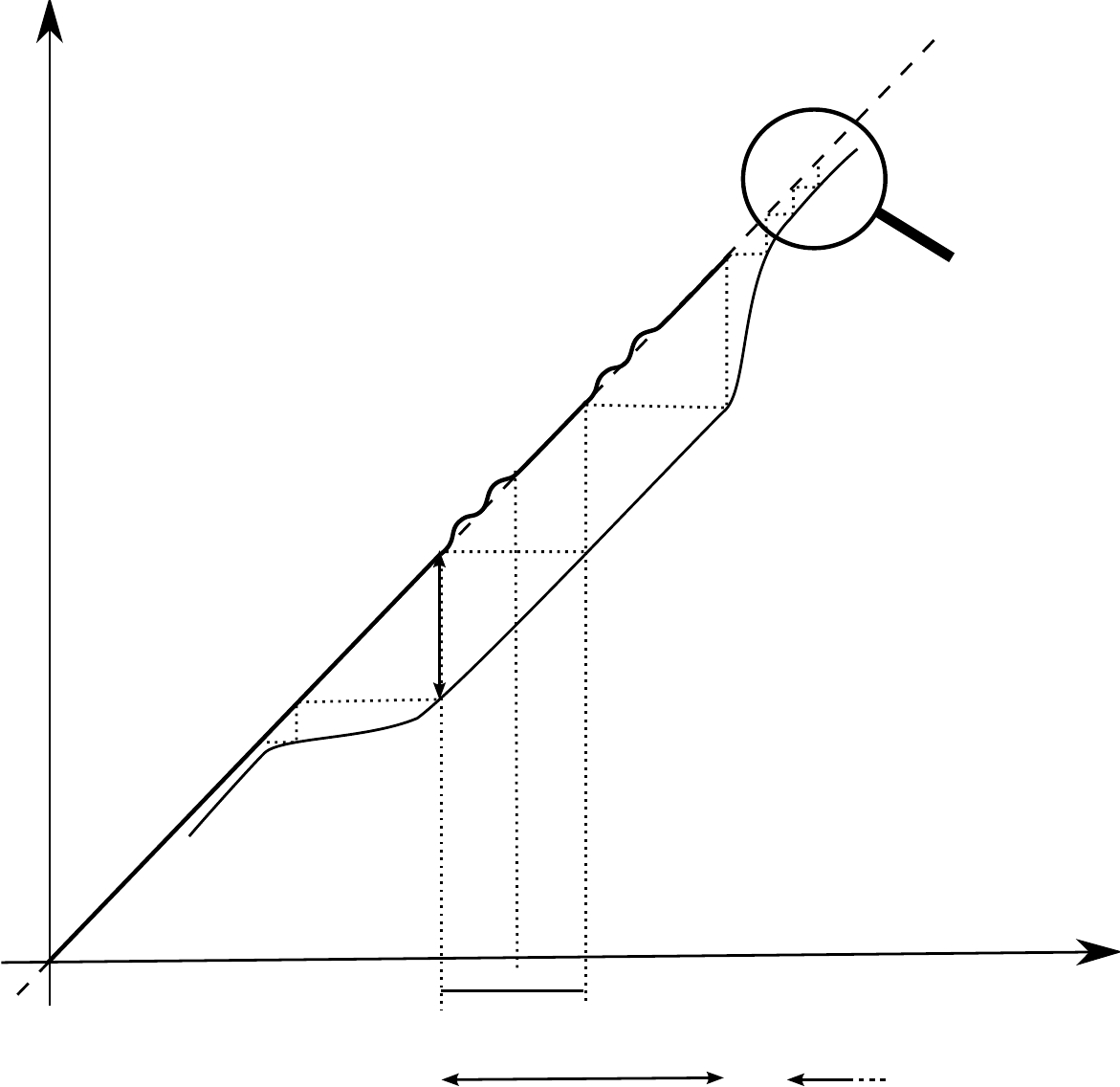}
\put(-95,5){$S_k$}
\put(-85,-9){$L_k$}
\put(-45,-9){$H_k$}
\put(-95,108){$g_k$}
\put(-60,95){$f_0^1$}
\put(-115,62){$v_k$}
\label{fig:gk}
\end{figure}

If $S_k^p \subset H_k$ on the other hand, the restriction of $f_0^p$ to $S_k^p$
is an affine map of the form 
$$ x \in S_k^p \mapsto \frac{ v_k }{ u_k }\; x + c_k, $$
where $c_k$ is a real constant. Hence, $g_k \res{S_k^p }$ is conjugate to $g_k
\res{S_k }$ by an affine map of huge ratio, precisely cooked up to make $g_k
\res{S_k^p }$ $\CC^2$ big ($g_k$ converges towards the identity in $\CC^1$
topology, though).
\begin{figure}[htbp]
\centering
\includegraphics[width=6cm]{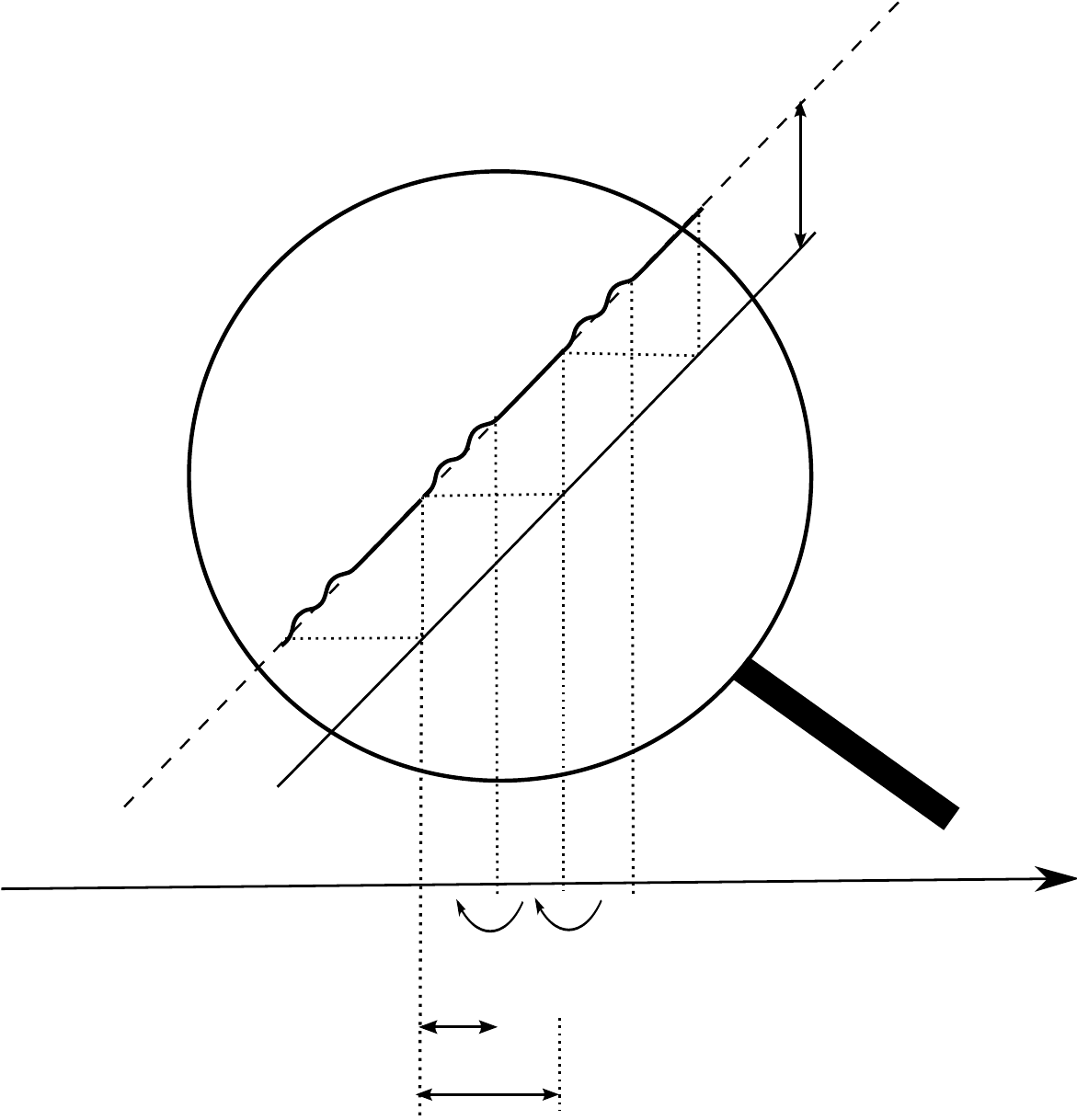}
\put(-102,6){$\scriptstyle{J_k^p}$}
\put(-97,-5){$\scriptstyle{S_k^p}$}
\put(-79,23){$f_0^{1/2}$}
\put(-42,150){$u_k$}
\put(-42,130){$f_0^1$}
\put(-110,105){$g_k$}
\label{fig:gkzoom2}
\end{figure}

The disymetric behaviour of $\gamma_k$ had a purpose as well: on one half of the
segment $S_k^p$, one can check that $g_k^{-1} \circ f_0^{1/2} \circ g_k -
f_0^{1/2}$ is exactly $g_k - \id$, and hence $\CC^2$ big. Superimposing all
these perturbations (\emph{i.e} conjugating by $h_k = g_k\circ ... \circ g_1$
and taking the $\CC^1$ limit) has the desired effect on the time-$1/2$ map of
the limit vector field.

\subsection{Combination with Anosov--Katok-type methods}

Now let $\alpha$ be an irrational number. We want to modify the above
construction so that in the end, both $1$ and $\alpha$ are smooth times of the
limit vector field. The idea is to pick an approximation of $\alpha$ by rational
numbers $p_k/q_k$, $k\ge 1$, to take an initial vector field $\nu_0$ similar to
Sergeraert's, and, this time, to ask $g_k$ to commute almost everywhere not with
$f_0^1$ anymore, but with $f_0^{1/q_k}$ (and thus with both $f_0^{p_k/q_k}$ and
$f_0^{q_k/q_k} = f_0^1$). More precisely, $g_k$ is still the identity near $0$,
but this time, it is of the form $\id + \gamma_k$ on a fondamental interval of
$f_0^{1/q_k}$ lying in $L_k$ (and thus of length $v_k/q_k$). Again, $\gamma_k$
must be chosen $\CC^k$ small.
\begin{figure}[h!]
\centering
\includegraphics[width=6cm]{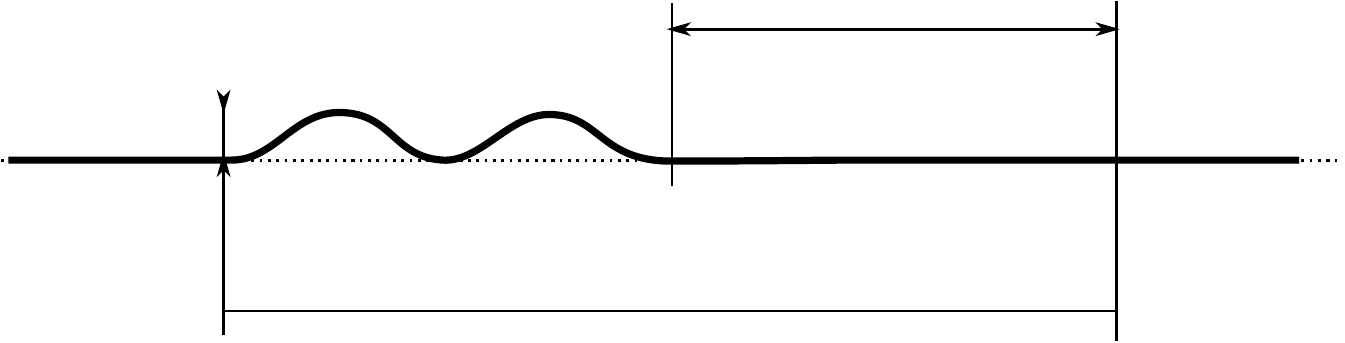}
\put(-90,-5){$v_k/q_k$}
\put(-130,36){$\gamma_k$}
\put(-155,25){$u_k$}
\put(-65,45){$v_k/2q_k$}
\label{fig:gammak}
\end{figure}
In particular, $u_k$ must be a $o(v_k^k/q_k^k)$. That way, one can make sure,
say, that
$$ \norm{f_k^{t} - f_{k-1}^{t}}_k = \norm{ g_k^{-1} \circ f_0^{t} \circ g_k -
f_0^t}_k = \norm{\gamma_k}_k< 2^{-k-1} \quad \text{for } t=p_k/q_k \;
\text{and}\; 1 $$
(both equalities are direct consequences of the construction). Now if $|\alpha -
p_k/q_k|$ is small enough (roughly speaking, $|\alpha - p_k/q_k| =
o(\norm{\nu_l}_k^{-1})$ for $l=k$ and $k-1$, assuming these ``norms'' are
finite), the above bounds remain true for $t=\alpha$ (replacing $2^{-k-1}$ by
$2^{-k}$, say), which ensures the regularity of the limit time-$\alpha$ map. But
based on the previous paragraph, the more $u_k = o(1/q_k^k)$ is small, the more
$\norm{g_k}_k$, $\norm{h_k}_k$ and thus $\norm{\nu_k}_k$ are big. So, basically,
in order for the process to converge, $|\alpha - p_k/q_k|$ must be much smaller
than $1/q_k^k$, and hence $\alpha$ must be a Liouville number.

In \cite{Ey1}, we proved the existence of some well-chosen $\alpha$ and $q_k$
for which the process indeed converges. The main contribution of this article is
to make all the ``rough" estimations above precise, \emph{i.e} to control the
size of the perturbations in terms of the initial data $q_k$, and to deduce from
it that \emph{any} Liouville number $\alpha$ has a suitable approximation by
rational numbers for which the process converges and provides the desired vector
field $\nu$.

\section{Notations and toolbox} \label{s:toolbox}

For any $\CC^k$ map $g$ on $\R_+$ we set
\begin{equation*}
   \Bars{ g }_k
= \sup \left\{ \bigbars{ D^mg (x) },\ 0 \le m \le k, \ x \in \R_+ \right\} \in [0,+\infty]. 
\end{equation*}
For any $g\in \DD^2$, we
define $Lf$ by 
\begin{equation*}
   Lf = D \log Df = \frac{ D^2 f }{ Df } .
\end{equation*}
The non-linear differential operator $L$ satisfies the following chain rule:
\begin{equation*} 
   L (h \circ g) = Lh \circ g \cdot Dg + Lg . 
\end{equation*}
To compute or control derivatives of products and compositions, we will also use
Leibniz rule:
$$D^k(gh) = \sum_{l=0}^{k} \binom{k}{l} D^lh \;D^{k-l}g$$
and Faà di Bruno's formula in the form
\begin{equation*} 
   D^k (h \circ g) = \sum_{\pi \in \Pi_k} 
   \left(D^{|\pi|} h\right) \circ g \cdot \prod_{B \in \pi} D^{|B|} g,
\end{equation*}
where $\Pi_k$ is the set of all partitions $\pi$ of $\{ 1, \cdots, k \}$ and 
$|X|$, for any finite set $X$, is the number of its elements. 

Finally, let $\eta$ be a vector field on $\R_+$. Throughout the paper, we will
make no difference between $\eta$ and the function $\eta / \partial_x$, where
$x$ is the underlying coordinate in $\R_+$, and in particular we will identify
$\partial_x$ with $1$. For $g\in\DD^1$, we denote by $g^*\eta$ the 
pullback of $\eta$ by $g$ which, viewed as a function, has the following
expression:
\begin{equation*}
 g^*\eta = \frac{ \eta \circ g }{ Dg }.
\end{equation*}

\section{A machine for turning rational approximations into vector fields}
\label{s:manu}

What we actually describe in this section is a ``manufacturing process" which,
to any increasing sequence of positive integers $(q_k)_{k \ge 1}$, associates a
specific $\CC^1$ vector field $\nu$ on $\R_+$, with a smooth time-$1$ map. Then
(in the next sections), we show that any Liouville number $\alpha$ has a
suitable approximation by rational numbers $(p_k/q_k)_{k \ge 1}$ such that the
vector field $\nu$ associated to the $q_k$'s has all the additional properties
listed in Theorem \hyperref[t:principalb]{A'}
.

Let $(q_k)_{k \ge 1}$ be any increasing sequence of positive integers (fixed
until the end of Section \ref{s:manu}). In order to produce $\nu$, we must first
associate to $(q_k)_{k \ge 1}$ a number of intermediate objects, the main of
which being an initial vector field $\nu_0$, smooth on $\R_+$, and a sequence
$(g_k)_{k \ge1}$ of smooth diffeomorphisms of $\R_+$. Those are used to deform
$\nu_0$ gradually into new smooth vector fields
$$\nu_k = h_k^*\nu_0 \quad \text{where $h_k = g_k\circ ... \circ g_1$},$$
which converge in $\CC^1$ topology, and we define $\nu$ as their limit. 

\subsection{Common basis}

Some material used to construct $\nu_0$ is common to every sequence $(q_k)_{k\ge 1}$,
namely the coefficients $(v_n)_{n\ge 1}$ defined by
$$v_n = 2^{-(n+3)^2} \quad\text{for all $n \ge 1$,}$$
and three smooth functions $\alpha, \beta, \gamma \from \R \to [0,1]$ 
satisfying the following conditions: 
\begin{itemize}
\item
$\alpha$ vanishes on $\left( -\infty, \frac 1 8 \right]$, equals $1$ on $\left[
\frac 1 4 , +\infty \right)$, and $\norm{\alpha}_1 < 16$;
\item
$\beta$ vanishes outside $\left[ \frac 1 8, \frac 7 8 \right]$, equals $1$ on 
$\left[ \frac 1 4, \frac 3 4 \right]$, and $\norm{\beta}_1 < 16$;
\item $\gamma$ vanishes outside $\left[ \frac 1 4 , \frac 3 4 \right]$, 
$\gamma(x) = x^2/2$ if $|x| \le 1/20$, and $\norm{\gamma}_1 < 1$.
\end{itemize}

\begin{figure}[htbp]
\centering
\includegraphics[width=12cm]{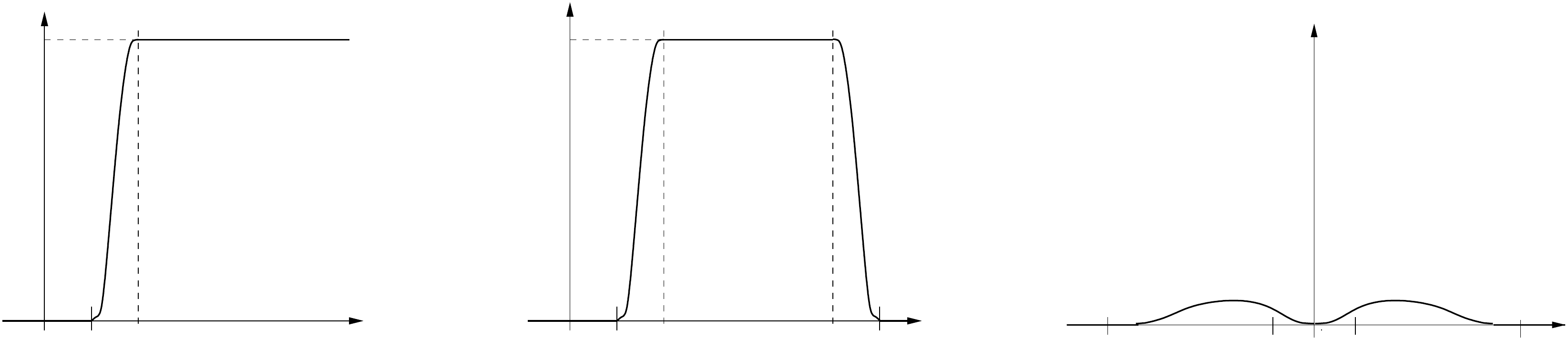}
\put(-13,-7){${\scriptstyle \frac 14}$}
\put(-110,-7){${\scriptstyle -\frac 14}$}
\put(-50,-7){${\scriptstyle \frac 1{20}}$}
\put(-74,-7){${\scriptstyle - \frac 1{20}}$}
\put(-80,20){${\scriptstyle \gamma}$}
\put(-200,-7){${\scriptstyle \frac 14}$}
\put(-162,-7){${\scriptstyle \frac 34}$}
\put(-210,-7){${\scriptstyle \frac 18}$}
\put(-212,20){${\scriptstyle \beta}$}
\put(-151,-7){${\scriptstyle \frac 78}$}
\put(-313,-7){${\scriptstyle \frac 14}$}
\put(-324,-7){${\scriptstyle \frac 18}$}
\put(-328,20){${\scriptstyle \alpha}$}
\label{fig:abg}
\end{figure}

\subsection{Initial vector field and related objects}

The coefficients $(u_n)_{n\ge1}$ defined now on the other hand, depend on
$(q_k)_k$:
\begin{equation}\label{e:uk}
u_n = 2^{-n-4} \; q_n^{-n} \; v_n^{n} \norm{\gamma}_{n}^{-1} \quad
\text{for all $n\ge 1$}.
\end{equation}
The initial vector field $\nu_0$ is then defined by:
\begin{gather}\label{e:nu0}
   \nu_0(x) = - u_{n+1} - (u_n - u_{n+1})\; \alpha (2^{n+1} x - 1)
 - (v_n - u_n)\; \beta (2^{n+1} x - 1) \\ 
   \text{for $x \in [2^{-n-1}, 2^{-n}]$, $n\ge 1$}, \quad 
   \nu_0(0) = 0 \quad \text{and} \quad
   \nu_0(x) = -u_1 \quad \text{for $x \ge 1/2$.} 
\end{gather}\vspace{-0.5cm}

\begin{figure}[htbp]
\centering
\includegraphics[width=10cm]{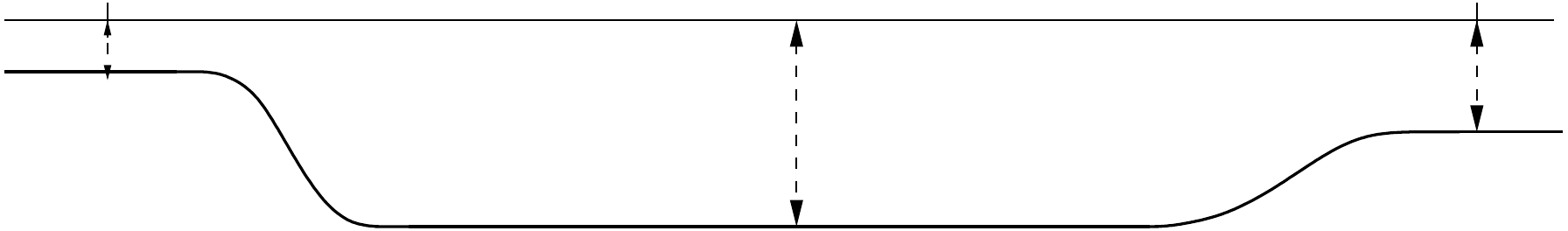}
\put(-285,32){${\scriptstyle u_{n+1}}$}
\put(-270,45){${\scriptstyle 2^{-n-1}}$}
\put(-152,18){${\scriptstyle v_{n}}$}
\put(-20,45){${\scriptstyle 2^n}$}
\put(-12,26){${\scriptstyle u_{n}}$}
\label{fig:nu_0}
\end{figure}

\noindent One easily checks that $\nu_0$ is smooth, infinitely flat at the 
origin and $\CC^1$-bounded ---~with $0 < \Bars{\nu_0}_1 < 1$. Furthermore,
$\nu_0$ equals $-v_n$ identically on the central part of $[2^{-n-1}, 2^{-n}]$, 
namely $[2^{-n-1} + 2^{-n-3}, 2^{-n} - 2^{-n-3}]$, and $-u_n$ on $[2^{-n} - 
2^{-n-4}, 2^{-n} + 2^{-n-3}]$. 

We denote by $\{ f_0^t, t \in \R \}$ the flow of $\nu_0$, and fix a forward
orbit $\{ a_l, l \ge 0 \}$ of $f_0 = f_0^1$, where $a_0 = 1$ and $a_l =
f_0(a_{l-1}) $ for all $l \ge 1$. A simple computation of travel time at
constant speed shows that for every $n \ge 1$, there exist integers $i$ and $j$
such that
\begin{align} \label{e:ain}
2^{-n} - {2^{-n-4}} & \le a_{i+2} < a_{i-1} 
\le 2^{-n} + {2^{-n-3}} \\
\llap{\text{and} \quad}
2^{-n-1} + 2^{-n-3} & \le a_{j+2} < a_{j-1} \le
2^{-n} - {2^{-n-3}} . \label{e:ajn}
\end{align}
We denote by $i(n)$ (resp. $j(n)$) the smallest integer $i$ (resp. $j$)
satisfying \eqref{e:ain} (resp. \eqref{e:ajn}). Thus $\nu_0$ equals $-v_n$ on
$[a_{ j(n)+2 }, a_{ j(n)-1 }]$, and hence $f_0^t$ induces on $[a_{ j(n)+1
},\linebreak[1] a_{ j(n)-1 }]$ the translation by $-tv_n$ for $0 \le t \le 1$.
Similarly, $f_0^t$ induces the translation by $-t u_n$ in a neighbourhood of
$a_{ i(n) }$.

\subsection{Conjugating diffeomorphisms and their properties}

For all $k\ge 1$, we define $\gamma_k : \R_+ \to [0,1]$ by:
\begin{equation} \label{e:gamma}
\gamma_k (x) = u_k \gamma \Bigl( \frac{q_k}{v_k} \bigl( x - a_{j(k)} \bigr)
\Bigr) \qquad \text{for all $x \in \R_+$.} 
\end{equation}
The map $\gamma_k$ is supported in $S_k = \left[ a_{j(k)} - \frac{v_k}{4q_k},
a_{j(k)} + \frac{v_k}{4q_k} \right]$, which is a fundamental interval of
$f_0^{1/2q_k}$ since it lies inside $[a_{j(k)+1}, a_{j(k)-1}]$ where the flow
$f_0^s$ of $\nu_0$ at time $0 \le s \le 1$ coincides with the translation by $-s
v_k$. Furthermore, for all $x \in \R_+$ and all $m \in \N$
\begin{align*} 
D^m\gamma_k (x) & = u_k\left(\frac{ q_k}{v_k}\right)^m 
D^m \gamma \Bigl( \frac{q_k}{v_k} \bigl( x - a_{j(k)} \bigr) \Bigr)\\
&= 2^{-k-4} \left( \frac{ q_k}{v_k} \right)^{m - k} \norm{\gamma}_k^{-1} 
D^m \gamma \Bigl( \frac{q_k}{v_k} \bigl( x - a_{j(k)} \bigr)\Bigr)
\end{align*}
by definition \eqref{e:uk} of $u_k$. In particular, 
\begin{equation}\label{e:gammak}
\Bars{ \gamma_k }_k = 2^{-k-4}.
\end{equation}
Now let $J_k$ denote the fundamental interval $\left[ a_{j(k)} -
\frac{v_k}{4q_k}, a_{j(k)} + \frac{3v_k}{4q_k} \right]$ of $f_0^{1/q_k}$. We
define $ g_k \from \R_+ \to \R_+$ as the unique map satisfying:
\begin{itemize}
\item
$g_k = \id$ on $\left[0 ,  a_{j(k)} - \frac{v_k}{4q_k}\right]$; 
\item
$g_k = \id + \gamma_k$ on $J_k$;
\item
$g_k$ commutes with $f_0^{1/q_k}$ outside $J_k$, so that
\begin{equation}
\label{e:gqn}
g_k = f_0^{-p/q_k} \circ (\id + \gamma_k) \circ f_0^{p/q} \quad \text{
on } f_0^{-p/q_k} \left( J_k \right) \; \text{for all } p \ge 0.
\end{equation}
\end{itemize}
In particular, all segments $f_0^{-p/q_k} \left( J_k \right), \; p \in \Z$, are
stable under $g_k\; (9')$. We now list some key properties of $g_k$.\\

For all $0 \le p \le q_k$, $f_0^{-p/q_k}$ and $f_0^{p/q_k}$ coincide with the
translations by $ \frac{p}{q_k} v_k$ and $- \frac{p}{q_k} v_k$ on $J_k$ and
$f_0^{-p/q_k}(J_k)$ respectively, so that \eqref{e:gqn} becomes:
\begin{equation}
\label{e:gqnb}
g_k = \id + \gamma_k\circ \left( \id - \frac{p}{q_k} v_k \right)\quad
\text{ on } f_0^{-p/q_k} \left( J_k \right), \; 0 \le p \le q_k.
\end{equation}
In particular, $g_k$ is the identity on
\begin{equation}\label{e:Nqn}
N_k = \bigcup_{p=0}^{q_k-1} \left( a_{j(k)} + (2p+1) \frac{v_k}{2q_k} +
\left[ - \frac{v_k}{4q_k},\frac{v_k}{4q_k} \right] \right), 
\end{equation}
and \emph{a fortiori} on every $f_0^{-p/q_k}(N_k)$, $p \ge 0$. This is also true
for $p<0$ since $g_k$ is the identity on $[0, a_{j(k)} - v_k/4q_k]$.

\begin{figure}[htbp]
\centering
\includegraphics[width=11cm]{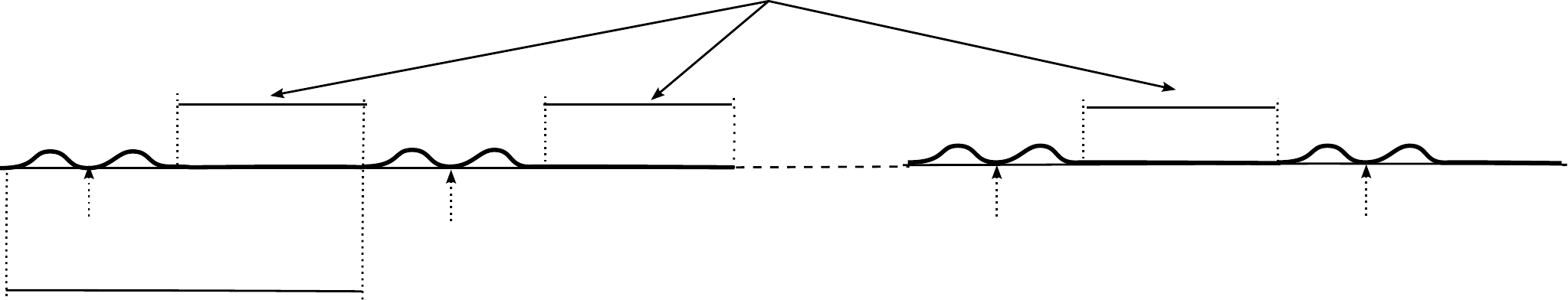}
\put(-160,64){${\scriptstyle N_k}$}
\put(-280,-8){${\scriptstyle J_k}$}
\put(-300,10){${\scriptstyle a_{j(k)}}$}
\put(-238,10){${\scriptstyle a_{j(k)}+\frac{v_k}{q_k}}$}
\put(-138,10){${\scriptstyle a_{j(k)}+(q_k-1)\frac{v_k}{q_k}}$}
\put(-50,10){${\scriptstyle a_{j(k)-1}}$}
\label{fig:Nk}
\end{figure}

Note furthermore that since $\nu_0$ is constant equal to $-u_1$ on
$[1/2,+\infty)$, $f_0^{-1/q_k}$ coincides with the translation by $u_1/q_k$ on
$[1/2,+\infty)$, so $g_k$ commutes with that translation there. \emph{A
fortiori}, $g_k$ commutes with the translation by $u_1$ on $[1,+\infty)$.
Furthermore,
\begin{align*}
g_k(a_0 = 1) = f_0^{-j(k)} & \circ g_k \circ f_0^{j(k)}(a_0) \\
& = f_0^{-j(k)}( g_k(a_{j(k)}) ) = f_0^{-j(k)}(a_{j(k)}) = a_0 =1,
\end{align*}
so $[1,+\infty)$ is stable under $g_k$.\medskip

After differentiation, \eqref{e:gqn} becomes
\begin{equation}\label{e:Dgqn1}
Dg_k = \frac{ Df_0^{p/q_k} }{ Df_0^{p/q_k} \circ g_k } \times 
\left(1 + D\gamma_k \circ f_0^{p/q_k}\right) \quad \text{on
$f_0^{-p/q_k} \left( J_k \right), \; p \ge 0$},
\end{equation}
so $g_k$ is a diffeomorphism since $\Bars{ \gamma_k }_1 < 1$, according to
\eqref{e:gammak}. One can actually simplify expression \eqref{e:Dgqn1}. The
vector field $\nu_0$ being invariant under the diffeomorphisms of its flow,
$$ Df_0^t=\frac{\nu_0 \circ f_0^t}{\nu_0}\quad \text{on $\R_+^* = (0,+\infty)$
for all $t \in \R$,}$$
so
$$ \frac{Df_0^{p/q_k}}{ Df_0^{p/q_k}\circ g_k} = 
\frac{\nu_0 \circ f_0^{p/q_k}}{\nu_0} \times 
\frac{\nu_0 \circ g_k}{\nu_0 \circ f_0^{p/q_k} \circ g_k}.$$
But for all $x \in f_0^{-p/q_k} \left( J_k \right)$, 
$$\nu_0 \circ f_0^{p/q_k}(x) = \nu_0\circ f_0^{p/q_k} \circ g_k(x) = -v_k$$
so
\begin{equation}\label{e:Dgqn2}
Dg_k = \frac{\nu_0 \circ g_k}{\nu_0} \times 
\left(1 + D\gamma_k \circ f_0^{p/q_k} \right) \quad \text{on
$f_0^{-p/q_k} \left( J_k \right), \; p \ge 0$}.
\end{equation}

We now define for all $k \ge 1$ a smooth diffeomorphism $h_k = g_k \circ ...
\circ g_1$ and a smooth vector field $\nu_k = h_k^*\nu_0$. The flow
$\{f_k^t,\;t\in\R\}$ of $\nu_k$ is well defined and consists of smooth
diffeomorphisms of $\R_+$ satisfying $f_k^t = h_k^{-1}\circ f_0^t \circ h_k$.
Note that $h_k$, like $g_l$ for all $l\le k$, commutes with the translation by
$u_1$ on $[1,+\infty)$. Let us define furthermore the (possibly empty) sets
$H_{k_0}$, for all $k_0\ge 1$, and $H$ by
\begin{equation}
\label{e:Hn0}
H_{k_0} = \bigcap_{l\ge k_0} \bigcup_{0 \le p < q_l} \left[ \frac {2p+1}{2q_l} -
\frac {1}{4q_l}, \frac {2p+1}{2q_l} + \frac {1}{4q_l} \right]
\end{equation}
and 
\begin{equation}\label{e:H}
H = \bigcup_{k_0 \ge1} H_{k_0}.
\end{equation}
We will need the following lemma in the proof of Proposition \ref{p:cv} (cf. \ref{ss:conv}) to show
that for all $t\in H$, the time-$t$ map of the limit vector field $\nu$ is not
$\CC^2$.

\begin{lemma} \label{l:Phi}
Let $ t\in H_{k_0} \subset H$ for some $k_0 \ge 1$. For all $k \ge k_0$, $h_k$
has the following behaviour on the orbits $\{a_n, \, n \in \Z\}$ and
$\{b_n=f_0^{-t}(a_n), \, n \in \Z\}$ of $f_0^1$\up:
\begin{enumerate}
\item 
$h_k$ is infinitly tangent to the identity at $b_n$ for all $n \ge j(k_0)$;
\item
$h_k$ is $\CC^1$-tangent to the identity on $\{a_n, \, n \in \Z\}$ ---~i.e $h_k(a_n) =
a_n$ and $Dh_k(a_n) = 1$ for all $n \in \Z$\up;
\item 
$(Lh_k - Lh_{k-1}) (a_n)$ equals $\frac{u_k q_k^2}{ v_k |\nu_0(a_n)| }$ if $n
\le j(k)$ and $0$ otherwise.
\end{enumerate}
\end{lemma}

\begin{proof}
Let $k \ge k_0$. To prove the first point, we must check that for all $l \ge 1$
and $n \ge j(n_0)$, $g_l$ is the identity near $b_n$. For $l < k_0$, this is true
because $b_n \notin [a_{j(l)} - \frac{v_l}{4q_l},+\infty)$, which contains the
support of $g_l$. As for $l \ge k_0$, according to \eqref{e:Nqn}, we only need to
check that $b_n\in f_0^p(N_l)$ for some $p\in \N$. But
$$b_{j(l)} = f_0^{-t}(a_{j(l)}) = a_{j(l)} + t v_{l} \in N_{l}$$
by definition of $H_{k_0}$, so $b_n = f_0^{n-j(l)}(b_{j(l)}) \in
f_0^{n-j(l)}(N_{l})$ for all $n\in \Z$, which concludes the proof of the first
point.\medskip

Now $\gamma(0) = D\gamma(0) = 0$, so $\gamma_l(a_{j(l)}) = D\gamma_l(a_{j(l)}) =
0$ for all $l \ge 1$, according to \eqref{e:gamma}, and since $g_l = \id +
\gamma_l$ on $J_l$, $g_l$ is tangent to the identity at $a_{j(l)}$. This is also
true at every point $f_0^{-p/q_l}(a_{j(l)}), \, p \ge 0$, by definition \eqref{e:gqn}
of $g_l$ (in particular at every $a_n$, $n \le j(l)$), and at every $a_n$, $n >
j(l)$ since $g_l = \id$ on a neighbourhood of $[0,a_{j(l)+1}] =
[0, a_{j(l)} - v_{l}]$. This, applied to $h_k = g_k \circ ... \circ g_1$, proves the
second point.\medskip

Let us now apply the chain rule to $h_k = g_k \circ h_{k-1}$:
\begin{equation*}
Lh_k = Lg_k \circ h_{k-1} \times Dh_{k-1} + Lh_{k-1}.
\end{equation*}
For all $n \in \Z$, point 2 tells us that $h_{k-1} (a_n) = a_n$ and $Dh_{k-1} (a_n) = 1$,
 so the above equality gives
$$ (Lh_k - Lh_{k-1}) (a_n) = Lg_k(a_n). $$
For $n > j(k)$, $Lg_k (a_n) = 0$ since $g_k$ is the identity on a neighbourhood of
$\left[ 0, a_{j(k)+1} \right]$. Suppose now that $n \le j(k)$ and write 
$p = j(n_k)-n \ge 0$. According to \eqref{e:gqn}, on a neighbourhood of $a_n$,
$g_k$ is given by:
\begin{equation*}
g_k = f_0^{-p} \circ (\id + \gamma_k) \circ f_0^p.
\end{equation*}
Furthermore
\begin{equation*}
\id = f_0^{-p} \circ \id \circ f_0^p.
\end{equation*}
The chain rule formula applied to both equalities gives:
\begin{align*}
Lg_k = Lg_k - L\id = Lf_0^{-p} & \circ \left( \id + \gamma_k \right) \circ
f_0^p \times D\left( \id + \gamma_k \right) \circ f_0^p \times Df_0^p\\ & +
L(\id + \gamma_k) \circ f_0^p \times Df_0^p- Lf_0^{-p} \circ f_0^p \times
Df_0^p.
\end{align*}
At $a_n = f_0^{-p}(a_{j(k)})$, we get
\begin{align*}
Lg_k(a_n) = Lf_0^{-p} & \left( a_{j(k)} + \gamma_k(a_{j(k)}) \right) \times
(1 + D\gamma_k(a_{j(k)}) ) \times Df_0^p(a_n)\\
& - Lf_0^{-p} (a_{j(k)}) \times Df_0^p(a_n)+
L(\id + \gamma_k)(a_{j(k)}) \times Df_0^p(a_n).
\end{align*}
Since $\gamma_k(a_{j(k)}) = D\gamma_k(a_{j(k)}) = 0$, the first two terms 
cancel each other. In the end, the invariance relation $\nu_0 \circ
f_0^p = Df_0^p \times \nu_0 $ applied at $a_n$ and the definition of $\gamma_k$ 
give 
$$Lg_k(a_n) = L(\id + \gamma_k)( a_{j(k)} ) \times \frac{ \nu_0 (a_{j(k)})}
{\nu_0(a_n)} = \frac{u_k q_k^2}{v_k^2}
\times\frac{v_{k}}{|\nu_0(a_n)|}.$$
\end{proof}

\subsection{Convergence of the deformation process and properties of the limit}
\label{ss:conv}
\def\esti{\mathrm{i}}

\begin{proposition}\label{p:cv}
For all $k\ge 1$, 
\begin{equation}
\bigBars{f_k^t - f_{k-1}^t }_k \le 2^{-k-4} \quad \text{for all} \quad t \in
\frac1{q_k} \Z \cap [0,1].\tag{$\esti_k$} \label{ik}
\end{equation}
In particular, the time-$1$ maps $f_k^1$ converge in $\Cinf$ topology towards a
smooth diffeomorphism $f$ with no other fixed point than $0$, whose Szekeres
vector field $\nu$ is the $\CC^1$ limit of the vector fields $\nu_k$. On the
other hand, for all $t$ in $H$, the time-$t$ map $f^t$ of $\nu$ is not $\CC^2$.
\end{proposition}

\begin{proof}
Let us start with estimate ($\esti_k$). Let $\{\varphi_k^t,\; t\in\R\}$ denote
the flow of $g_k^* \nu_0$, so that
\begin{equation*}
\varphi_k^t = g_k^{-1} \circ f_0^t \circ g_k. 
\end{equation*}
Since
\begin{align*}
\nu_k = h_k^* \nu_0 = h_{k-1}^* g_k^* \nu_0 \qquad & \text{and} \qquad 
\nu_{k-1} = h_{k-1}^* \nu_0, \\
\intertext{the flows of $\nu_k$ and $\nu_{k-1}$ are given by}
f_k^t = h_{k-1}^{-1} \circ \varphi_k^t \circ h_{k-1}
\quad & \text{and} \quad 
f_{k-1}^t = h_{k-1}^{-1} \circ f_0^t \circ h_{k-1} .
\end{align*}
By definition, $g_k$ commutes with $f_0^{1/{q_k}}$ outside $J_k$. As a
consequence, $g_k$ commutes with any iterate $f_0^{p/{q_k}}$, $p \ge 1$, outside
the interval 
$$ \bigcup_{q=0}^{p-1} f_0^{-q/q_k}(J_k). $$
Thus, $\varphi_k^{p/{q_k}}$ coincides with $f_0^{p/{q_k}}$ outside this
interval. In particular, for $0 \le p \le q_k$, since $f_0^s$ coincides with the
translation by $-s v_{k}$ on $[a_{j(k)} - v_{k}, a_{j(k)} + v_{k}]$ for all $0
\le s \le 1$, $\varphi_k^{p/{q_k}}$ coincides with $f_0^{p/{q_k}}$ outside
\begin{equation*} \label{e:Mk}
M_k  = \left[ a_{j(k)} - \frac {v_{k}}{4q_k}, a_{j(k)} + v_{k} - 
\frac{v_{k}}{4q_k} \right].
\end{equation*}
Moreover, for all $x \in J_k$, 
\begin{align*}
\varphi_k^{1/{q_k}} (x) 
& = g_k^{-1} \circ f_0^{1/{q_k}} \circ g_k(x)\\
& = g_k^{-1} \left( g_k(x) - \frac{v_{k}}{q_k} \right) \\
& = g_k^{-1} \left( x + \gamma_k(x) - \frac{v_{k}}{q_k} \right) \quad
\text{by definition of $g_k$ on $J_k$} \\
& = x - \frac{v_{k}}{q_k} + \gamma_k(x) \quad
\text{since $x + \gamma_k(x) - \frac{v_{k}}{q_k} < \min (\Supp g_k^{-1})$} \\
&= f_0^{1/{q_k}} (x) + \gamma_k(x). 
\end{align*}
Thus, since $\varphi_k^{1/{q_k}}$ coincides with $f_0^{1/{q_k}}$ outside
$J_k$, $\varphi_k^{1/{q_k}} - f_0^{1/{q_k}} = \gamma_k$ on all of $\R_+$.
Similarly, for all $0 \le p \le q_k$,
\begin{equation} \label{e:sigma}
\varphi_k^{p/ {q_k}} (x) - f_0^{ p/ {q_k} } (x)
 = \sum_{q=0}^{p-1} \gamma_k \left( x - \frac {qv_{k}} {q_k} \right) \qquad
\text{for all $x \in \R_+$,}
\end{equation} 
\begin{equation} \label{e:sigmanorm}
\text{so } \norm{ \varphi_k^{ p/ {q_k} } - f_0^{ p/ {q_k} } }_m
 = \norm{ \gamma_k }_m \quad\text{for all } m\in\N.
\end{equation}
But in the region $M_k$ where $\varphi_k^{p/{q_k}}$ and $f_0^{p/ {q_k}}$ differ
for $0 \le p \le q_k$, the diffeomorphism $h_{k-1}$ is the identity since
$$\Supp h_{k-1}\subset \bigcup_{l \le k-1} \Supp g_l \subset
\left[ a_{j({k-1})} - \frac{v_{{k-1}}}{4 q_{k-1}}, +\infty \right).$$
Consequently, for all $0 \le p \le q_k$, the relations
\begin{equation*}
f_k^{p/{q_k}} = h_{k-1}^{-1} \circ \varphi_k^{p/{q_k}} \circ h_{k-1} 
\end{equation*}
and
\begin{equation*}
f_{k-1}^{p/{q_k}} = h_{k-1}^{-1} \circ f_0^{p/{q_k}} \circ h_{k-1} 
\end{equation*}
imply:
\begin{equation}\label{e:fk-fk-1}
f_k^{p/{q_k}} - f^{p/{q_k}}_{k-1} = 
\begin{cases}
\varphi_k^{p/{q_k}} - f_0^{p/{q_k}} \quad\text{ on } M_k\\
0 \quad\text{ outside}, 
\end{cases}
\end{equation}
which, together with \eqref{e:sigmanorm}, gives \eqref{ik}:  
\begin{equation*}
\lrBars{f_k^{p/{q_k}} - f_{k-1}^{p/{q_k}}}_k
\le \lrBars{ \varphi_k^{p/ {q_k} } - f_0^{p/ {q_k} } }_k 
= \bigBars{ \gamma_k }_k 
\le 2^{-k-4}.
\end{equation*}
As a consequence, the time-$1$ maps $f_k^1 = f_k$ converge towards a smooth
diffeomorphism $f$. Let us note furthermore that
\begin{equation}\label{e:nonnul}
\lrbars{\frac{f_k(x) - f_{k-1}(x)}{f_0(x) - x}}\le 2^{-k-2} \quad\text{for all
$k\ge 1$}.
\end{equation}
Indeed, according to \eqref{e:sigma} and \eqref{e:fk-fk-1}, 
\begin{equation*}
f_k(x) - f_{k-1}(x) = 
\begin{cases}
\displaystyle{\sum_{q=0}^{q_k-1}} \gamma_k \left( x - \frac {qv_{k}} {q_k}
\right) \;& \text{ on } M_k,\\ \hspace{1cm} 0 &\text{ outside},
\end{cases}
\end{equation*}
so since at most one term of the above sum is nonzero, 
$$|f_k(x) - f_{k-1}(x)| \le \norm{\gamma_k}_0 \le u_k.$$
But on $M_k$, 
$$|f_0(x)-x| =  v_k.$$
The last two remarks imply inequality \eqref{e:nonnul} since $u_k/v_k\le 2^{-k-2}$. 
Thus for all $x \in \R_+^*$,
\begin{align*}
|f(x) - x | & = \left| f_0(x) - x + \sum_{k\ge 1} \big(f_k(x) - f_{k-1}(x)\big)\right| \\
& \ge |f_0(x) - x| \left(1 - \sum_{k \ge 1} 2^{-k-2}\right) \\
& \ge \frac{|f_0(x) - x|} 2 > 0.
\end{align*}
So $f$ has no other fixed point than $0$. 

We could prove the $\CC^1$ convergence of the vector fields $\nu_k$ by hand, as
in \cite{Ey1} and \cite{Ey2}. But since a third similar proof would be of little
interest, we choose to invoke a different argument here. In fact, the
convergence of the $\nu_k$ can be derived directly from the $\CC^\infty$
convergence of their time-$1$ maps, as an immediate consequence of a theorem by
J.-C. Yoccoz \cite[chap. 4, Theorem 2.5]{Yo} asserting the continuous dependence
of the Szekeres vector field with respect to its time-$1$ map (in a more general
setting and for suitably defined topologies). We denote by $\nu$ the limit of
$\nu_k$ and by $\{f^t,\; t\in \R\}$ the flow of $\nu$ (so that $f=f^1$). For all
$t\in \R$, $f^t$ is the limit of $f_k^t$ in $\CC^1$ topology.
\medskip

Now let $t \in H_{k_0}$ for some $k_0 \ge 1$. We want to prove that $Lf^{t}$ is
not continuous at $0$. To do that, we compute $Lf^t$ at $b_{ i(l) } =
f_0^{-t}(a_{i(l)})$ for all $l \ge k_0 +1$. By invariance of $\nu$ under its
flow,
$$Df^t = \frac{\nu \circ f^t}{\nu} \quad \text{on } \R_+^*$$
from which one computes
$$Lf^t = \frac{D\nu \circ f^t - D\nu}{\nu}.$$
In particular,
$$Lf^{t}( b_{i(l)} ) =
- \frac{ D\nu(f^{t}(b_{i(l)})) - D\nu(b_{i(l)}) }{u_{l}}.$$
But for all $k\ge k_0$, 
\begin{align*}
f_k^{t}(b_{i(l)}) & = h_k^{-1} \circ f_0^{t} \circ h_k(b_{i(l)}) \\
& = h_k^{-1} \circ f_0^{t}(b_{i(l)}) \quad \text{according to Lemma
\ref{l:Phi}},\\
& = h_k^{-1}(a_{i(l)}) = a_{i(l)} \quad \text{according to
Lemma \ref{l:Phi} again} ,
\end{align*}
so $f^{t}(b_{i(l)}) = \lim_k f_k^{t}(b_{i(l)}) = a_{i(l)}$. Besides, the derivative of 
$\nu_k = h_k^*\nu_0$ is
$$D\nu_k = D\nu_0 \circ h_k - (\nu_0 \circ h_k) \frac{Lh_k}{Dh_k},$$
so for all $k \ge l$, according to points 2 and 3 of Lemma
\ref{l:Phi},
\begin{equation}\label{e:Dnuai}
D\nu_k (a_{i(l)}) = D\nu_0 (a_{i(l)}) -
\nu_0(a_{i(l)}) Lh_k (a_{i(l)})
 = \sum_{n=l}^k \frac{u_{n} q_n^2}{v_{n}},
\end{equation}
and according to point 1 of the same lemma,
\begin{equation}\label{e:Dnubi}
D\nu_k (b_{i(l)}) = D\nu_0 (b_{i(l)}) - \frac{Lh_k}{Dh_k}(b_{i(l)})
\nu_0(b_{i(n_l)}) = 0 - 0 = 0.
\end{equation}

The vector fields $\nu_k$ converge towards $\nu$ in $\CC^1$ topology on $\R_+$, so 
Formulae \eqref{e:Dnuai} and \eqref{e:Dnubi} give
\begin{equation*}
D\nu(a_{i(l)}) = \sum_{n\ge l} \frac{u_{n} q_n^2}{v_{n}}\quad
\text{and}\quad D\nu(b_{i(l)}) = 0.
\end{equation*}
In the end, 
$$Lf^{t}(b_{i(l)}) = -\sum_{n\ge l} \frac{u_{n} q_n^2}{v_{n} u_{l}} <
-\frac{q_l^2}{v_{l}} \to - \infty \quad [l \to \infty]$$
so $f^{t}$ is not $\CC^2$ at $0$.
\end{proof}

\section{Polynomial control of the manufactured objects}

\begin{proposition}\label{l:klemma}
There are maps $n$ and $c \from \N^2 \to \N^*$ such that for any
increasing sequence $(q_k)_{k \ge 1}$ of positive integers, the vector fields
$(\nu_k)_{k \ge 0}$ built from $(q_k)_{k \ge 1}$ and their flows
$\{f_k^t,\;t \in\R\}$ satisfy
\begin{equation}\label{e:klemma}
\norm{\nu_k \circ f_k^t}_r \le c(k,r) q_k^{n(k,r)} \quad \text{for all $(k,r) \in
\N^2$ $($with $q_0:=1)$.}
\end{equation}

\end{proposition}

This proposition relies on the following assertions.

\begin{lemma}\label{l:nu0}
There are universal bounds on all derivatives of $\nu_0$ and $f_0^t$, $t\in
[0,1]$, \emph{i.e.} bounds which depend neither on $(q_k)_k$ nor on $t$.
\end{lemma}

\begin{lemma}\label{l:gk}
There is a polynomial \up(in $q_k$\up) control on the growth of the derivatives
of $g_k$, \emph{i.e.} there exist universal maps $c, n \from \N^*\times \N \to
\N^*$ such that for any $(q_k)_{k \ge 1}$, the associated $(g_k)_{k \ge 1}$
satisfies
\begin{equation}\label{e:bound}
\max \left( \Bars{g_k - \id}_r, \Bars{g_k^{-1} - \id}_r \right)< c(k,r) q_k^{n(k,r)}
\end{equation}
for all  $(k,r) \in \N^* \times \N$.
\end{lemma}

\begin{proof}[Proof of Proposition \ref{l:klemma} using Lemmas \ref{l:nu0} and
\ref{l:gk}]
We proceed by induction on $k$. Step $k=0$ follows directly from Lemma
\ref{l:nu0} and Fa\`a di Bruno's Formula. For $k \ge 1$, step $k$ follows from step $k-1$ and Lemma \ref{l:gk} applying Faà di Bruno's and
Leibnitz' derivation formulas to the relations
$$\nu_k = g_k^*\nu_{k-1} = (\nu_{k-1} \circ g_k)(Dg_k^{-1} \circ g_k) \quad 
\text{and} \quad f_k^t = g_k^{-1} \circ f_{k-1}^t \circ g_k.$$
\end{proof}

\begin{proof}[Proof of Lemma \ref{l:nu0}]
It is clear from the definition \eqref{e:nu0} of $\nu_0$ that its derivatives
are bounded independently of the coefficients $(u_n)_n$, and thus of $(q_n)_n$.
Similar bounds on the derivatives of the flow (for a compact set of times) are
then easilly derived from an appropriate (generalized) version of Gronwall's
Lemma.
\end{proof}

\begin{proof}[Proof of Lemma \ref{l:gk}]
Let $k \ge 1$. The orders $r=0$ and $r=1$ are easily settled using $(9')$,
\eqref{e:gammak} and \eqref{e:Dgqn2}. In particular,
\begin{equation} \label{e:gk-id}
\Bars{g_k - \id}_1 < \frac 1 2 \quad \text{for all $k$.}
\end{equation}
Note that given \eqref{e:gk-id}, a polynomial (in $q_k$) control on the growth
of the derivatives of $g_k-\id$ automatically gives one on $g_k^{-1}-\id$. This
is because the inverse of any smooth diffeomorphism $g$ satisfies
\begin{equation}\label{e:inverse}
(D^r g^{-1}) \circ g= \frac {P_r(Dg,...,D^rg)}{(Dg)^{2r+1}},
\end{equation}
where $P_r$ is a universal polynomial in $r$ variables (independent of $g$), and
in our case, $Dg = Dg_k$ is bounded below independently of $(q_n)_n$. Formula
\eqref{e:inverse} is obtained by induction on $r$, starting with the identity
$Dg^{-1} \circ g \times Dg = 1$ and using Faà di Bruno's Formula.

We now focus on $g_k-\id$. Recall that
\begin{equation}\label{e:gkrecall}
g_k = \begin{cases} \hspace{1.4cm} \id & \text{on $[0, \min J_k]$}\\
\hspace{1cm} \id + \gamma_k & \text{on $J_k$}\\
f_0^{-p} \circ (\id + \gamma_k) \circ f_0^p & \text{on $f_0^{-p}(J_k)$, for all
$p\ge 1$.}
\end{cases}
\end{equation}
Thus, on $[0, \max J_k]$, 
$$\left| D^r(g_k - \id) \right| = \lrbars{D^r\gamma_k} 
\le u_k \left( \frac{q_k}{v_k} \right)^r \norm{\gamma}_r 
\le c(r,k) q_k^{n(r,k)},  $$
with 
$$ c(r,k ) = \frac {2^{-k-4} \norm{\gamma}_r v_k^{k-r}}{\norm{\gamma}_{k}}
\quad \text{and} \quad n(r,k) = r-k,$$
by definition \eqref{e:uk} of $u_k$. Then, given \eqref{e:gkrecall} 
(and Faà di Bruno's formula again), a uniform (in $p$) polynomial (in $q_k$) control 
on the derivatives of $f_0^p \res {f_0^{-p}(J_k)}$ is sufficient to ensure the desired control on 
$D^r(g_k - \id)$ on the rest of $\R_+$. 

The vector field $\nu_0$ being preserved by its own flow, 
$$Df_0^p = \frac{\nu_0 \circ f_0^p }{\nu_0} \quad \text{on $\R_+^*$}.$$
In particular, on $f_0^{-p}(J_k)$,
$$Df_0^p = -\frac{v_k}{\nu_0}, $$
and thus, for all $r \ge 1$,
\begin{equation}\label{e:Df0p}
D^{r+1}f_0^p = \frac{Q_r(\nu_0,...,D^r\nu_0)}{\nu_0^{2^r}},
\end{equation}
where $Q_r$ is a universal polynomial (independent of $\nu_0$) in $r+1$
variables. According to Lemma \ref{l:nu0}, for each $r$, the numerator of
\eqref{e:Df0p} is bounded independently of $(q_k)_k$. As for the denominator,
$|\nu_0(x)| \ge u_k$ for all $x \in [\max J_k,\infty)$, so by definition
\eqref{e:uk} of $u_k$,
$$\frac{1}{\nu_0^{2^r} } \le \left(2^{k+4}v_k^{-k} 
\norm{\gamma}_{k} \right)^{2^r} q_k^{2^r(k+1)},$$
which is the kind of control we were looking for (the bound does not depend on $p$).
\end{proof}

\section{Convergence of the time-$\alpha$ maps}
\label{s:conv}

\begin{proposition}\label{p:principal}
Let $\alpha$ be a Liouville number. There is a sequence $(p_k/q_k)_{k \ge 1}$
of rational approximations of $\alpha$ such that the vector field $\nu$ built
from $(q_k)_{k \ge 1}$ has all the properties described in Theorem
\hyperref[t:principalb]{A'}.
\end{proposition}

Let $\alpha$ be a Liouville number. By definition, there exists a sequence
$(p_k/q_k)_{k \ge 1}$ of rational approximations of $\alpha$ satisfying
\begin{equation}\label{e:C}
\left| \alpha - \frac{p_k}{q_k} \right| < \frac{2^{-k-2} c(k,k)^{-1} }{q_k^{n(k,k)}}
\quad \text{for all $k\ge 1$} \tag{$C_k$}
\end{equation}
(where $c$ and $n$ are the maps given by Proposition \ref{l:klemma}), with the additional requirement that
\begin{equation} \label{e:C'}
\frac1{q_{k+1}}< \frac{2^{-k-2} c(k,k)^{-1} }{q_k^{n(k,k)}}\quad\text{for all $k\ge 1$},\tag{$C_k'$}
\end{equation}
so that every segment $\left[\frac{p}{q_k}- \frac{2^{-k-2} c(k,k)^{-1} }{q_k^{n(k,k)}}, \frac{p}{q_k}+\frac{2^{-k-2} c(k,k)^{-1} }{q_k^{n(k,k)}}\right]$, $p\in\Z$, contains at least two elements of $\frac{1}{q_{k+1}}\Z$, making 
\begin{equation}\label{e:K}
K'= \bigcap_{k\ge 1}\bigcup_{0\le p\le q_k} \left[\frac{p}{q_k}- \frac{2^{-k-2} c(k,k)^{-1} }{q_k^{n(k,k)}}, \frac{p}{q_k}+\frac{2^{-k-2} c(k,k)^{-1} }{q_k^{n(k,k)}}\right]
\end{equation}
a Cantor set, with $\alpha\in K:=K'+[\alpha]$ (where $[\alpha]$ denotes the integral part of $\alpha$). Similarly, for such a sequence $(q_k)_k$, the set $H$ defined by \eqref{e:H} is a Cantor
set (in particular nonempty). Hence, Proposition \ref{p:principal}, and thus
Theorem \hyperref[t:principalb]{A'}, follow from Lemma \ref{l:alpha} below and Proposition
\ref{p:cv}.
\begin{lemma}\label{l:alpha}
Let $\alpha$ be a Liouville number, $(p_k/q_k)_{k \ge 1}$ a sequence
of rational approximations of $\alpha$ satisfying \eqref{e:C} and \eqref{e:C'} for all $k \ge 1$, and $K'$ the Cantor set defined by \eqref{e:K}.
Then the vector fields $\nu_k$ associated to $(q_k)_{k \ge 1}$ and their flows
satisfy
\begin{equation}
\bigBars{f_k^\tau - f_{k-1}^\tau }_k \le 2^{-k} \quad \text{for all $k\ge 1$ and $\tau\in K'$}.
\end{equation}
As a consequence, the time-$\tau$ maps of the limit $\nu$ of $\nu_k$ are smooth for all $\tau\in K'$.
\end{lemma}

\begin{proof}
Let $\tau\in K'$ and $(r_k)_{k\ge 1}$ the sequence of integers such that 
\begin{equation}
\tau\in  \left[\frac{r_k}{q_k}- \frac{2^{-k-2} c(k,k)^{-1} }{q_k^{n(k,k)}}, \frac{r_k}{q_k}+\frac{2^{-k-2} c(k,k)^{-1} }{q_k^{n(k,k)}}\right]\quad\text{for all $k\ge 1$}.
\end{equation}
Let $k\ge 1$. 
$$ \bigBars{ f_k^\tau - f_{k-1}^\tau }_k \le \bigBars{ f_k^\tau -
f_{k}^{r_k/q_k} }_k + \bigBars{ f_k^{r_k/q_k} - f_{k-1}^{r_k/q_k} }_k +
\bigBars{f_{k-1}^{r_k/q_k} - f_{k-1}^\tau }_k.$$
According to \eqref{ik} in Proposition \ref{p:cv}, the central term is less than
$2^{-k-4}$. Now
$$D^n\left( f_k^\tau - f_{k}^{r_k/q_k} \right) = D^n\left(\int_{r_k/q_k}^\tau 
\frac{df_k^t}{dt} dt \right) 
= \int_{r_k/q_k}^\tau D^n(\nu_k \circ f_k^t) dt,  $$
so 
$$ \bigBars{f_k^\tau - f_{k}^{r_k/q_k} }_k \le \left| \tau -
\frac{r_k}{q_k} \right| \bigBars{\nu_k \circ f_k^t }_k \le 2^{-k-2}$$
according to \eqref{e:C} and Proposition \ref{l:klemma}. A similar argument gives 
$$ \bigBars{f_{k-1}^{r_k/q_k}  - f_{k-1}^\tau}_k \le 2^{-k-2}$$
and in the end,
\begin{equation*}
\bigBars{f_k^\tau - f_{k-1}^\tau }_k \le 2^{-k}.
\end{equation*}
\end{proof}

\end{document}